\documentclass[12pt]{amsart}
\usepackage{amscd,amsmath,amsthm,amssymb}
\usepackage{pstcol,pst-plot,pst-3d}
\usepackage{lmodern,pst-node}
\usepackage{multicol}
 \usepackage[dvips]{graphicx}
 \usepackage{epstopdf}
 \usepackage{graphicx}
\usepackage{amsfonts,amssymb,amscd,amsmath,enumerate,verbatim}
\psset{unit=0.7cm,linewidth=0.8pt,arrowsize=2.5pt 4}

\newpsstyle{fatline}{linewidth=1.5pt}
\newpsstyle{fyp}{fillstyle=solid,fillcolor=verylight}
\definecolor{verylight}{gray}{0.97}
\definecolor{light}{gray}{0.9}
\definecolor{medium}{gray}{0.85}

\def\NZQ{\mathbb}               
\def\NN{{\NZQ N}}

\def\ZZ{{\NZQ Z}}

\def\frk{\mathfrak}               

\def\Phi{{\frk N}}
\def\Pc{{\mathcal P}}

\def\opn#1#2{\def#1{\operatorname{#2}}} 
\opn\chara{char} \opn\length{\ell} \opn\pd{pd} \opn\rk{rk}
\opn\projdim{proj\,dim} \opn\injdim{inj\,dim} \opn\rank{rank}
\opn\depth{depth} \opn\grade{grade} \opn\height{height}
\opn\size{size}
\opn\embdim{emb\,dim} \opn\codim{codim}

\opn\Tr{Tr} \opn\bigrank{big\,rank}
\opn\superheight{superheight}\opn\lcm{lcm}
\opn\trdeg{tr\,deg}
\opn\reg{reg} \opn\lreg{lreg} \opn\ini{in} \opn\lpd{lpd}
\opn\size{size}\opn{\mult}{mult}
\opn{\Cl}{Cl}\opn{\trdeg}{trdeg}
\opn\div{div} \opn\Div{Div} \opn\cl{cl} \opn\Cl{Cl}
\opn\Spec{Spec} \opn\Supp{Supp} \opn\supp{supp} \opn\Sing{Sing}
\opn\Ass{Ass} \opn\Min{Min} \opn\cl{cl}
\opn\Ann{Ann} \opn\Rad{Rad} \opn\Soc{Soc}
\opn\Syz{Syz} \opn\Im{Im} \opn\Ker{Ker} \opn\Coker{Coker}
\opn\Am{Am} \opn\Hom{Hom} \opn\Tor{Tor} \opn\Ext{Ext}
\opn\End{End} \opn\Aut{Aut} \opn\id{id} \opn\ini{in}

\opn\nat{nat}
\opn\pff{pf}
\opn\Pf{Pf} \opn\GL{GL} \opn\SL{SL} \opn\mod{mod} \opn\ord{ord}
\opn\Gin{Gin}
\opn\Hilb{Hilb}\opn\adeg{adeg}\opn\std{std}\opn\ip{infpt}
\opn\Pol{Pol}
\opn\sat{sat}
\opn\Var{Var}
\opn\Gen{Gen}
\opn\lex{lex}
\opn\div{div}

\opn\aff{aff} \opn\con{conv} \opn\relint{relint} \opn\st{st}
\opn\lk{lk} \opn\cn{cn} \opn\core{core} \opn\vol{vol}
\opn\link{link} \opn\star{star}
\opn\gr{gr}

\def\Ic{{\mathcal I}}

\def\Cc{{\mathcal C}}
\def\Dc{{\mathcal D}}

\def\pot#1#2{#1[\kern-0.28ex[#2]\kern-0.28ex]}

\opn\dirlim{\underrightarrow{\lim}}
\opn\inivlim{\underleftarrow{\lim}}
\let\sect=\cap

\let\Dirsum=\bigoplus

\let\to=\rightarrow

\def\Implies{\ifmmode\Longrightarrow \else
        \unskip${}\Longrightarrow{}$\ignorespaces\fi}
\def\implies{\ifmmode\Rightarrow \else
        \unskip${}\Rightarrow{}$\ignorespaces\fi}
\def\iff{\ifmmode\Longleftrightarrow \else
        \unskip${}\Longleftrightarrow{}$\ignorespaces\fi}

\let\:=\colon
\newtheorem{Theorem}{Theorem}[section]
\newtheorem{Lemma}[Theorem]{Lemma}
\newtheorem{Corollary}[Theorem]{Corollary}
\newtheorem{Proposition}[Theorem]{Proposition}

\let\epsilon\varepsilon
\let\phi=\varphi
\let\kappa=\varkappa
\def\qed{\ifhmode\textqed\fi
      \ifmmode\ifinner\quad\qedsymbol\else\dispqed\fi\fi}
\def\textqed{\unskip\nobreak\penalty50
       \hskip2em\hbox{}\nobreak\hfil\qedsymbol
       \parfillskip=0pt \finalhyphendemerits=0}
\def\dispqed{\rlap{\qquad\qedsymbol}}

\opn\dis{dis}
\def\pnt{{\raise0.5mm\hbox{\large\bf.}}}

\opn\Lex{Lex}
\opn\int{int}

\newcommand{\inD}[1][\relax]{\def\argone{#1}\def\temprelax{\relax}
  \ifx\argone\temprelax\right.\else\,\middle|#1\right.{}\fi}

\begin{document}
\title {Gr\"obner Bases of balanced polyominoes}
\author {J\"urgen Herzog, Ayesha Asloob Qureshi and Akihiro Shikama}
\thanks{This paper was partially written during the visit of the  second and third author at Universit\"at Duisburg-Essen, Campus Essen. The second  author wants to thank the Abdus Salam International Centre for Theoretical Physics (ICTP), Trieste, Italy  for  supporting her. The third author  wants to thank Professor Hibi who made his visit to Essen possible.
}


\address{J\"urgen Herzog, Fachbereich Mathematik, Universit\"at Duisburg-Essen, Campus Essen, 45117
Essen, Germany} \email{juergen.herzog@uni-essen.de}

\address{Ayesha Asloob Qureshi, The Abdus Salam International Center of Theoretical Physics, Trieste, Italy}
\email{ayesqi@gmail.com}
\address{Akihiro Shikama, Department of Pure and Applied Mathematics, Graduate School of Information Science and Technology,
Osaka University, Toyonaka, Osaka 560-0043, Japan}
\email{a-shikama@cr.math.sci.osaka-u.ac.jp}

\begin{abstract}
We introduce balanced  polyominoes and show that their ideal of inner minors is a prime ideal and has a squarefree Gr\"obner basis with respect to any monomial order, and we  show that any row or column convex and any tree-like polyomino is simple and balanced.
\end{abstract}
\maketitle

\section*{Introduction}
Polyominoes are, roughly speaking,  plane figures obtained by joining squares of equal size edge to edge. Their appearance origins in  recreational mathematics but also has been subject of many combinatorial investigations including tiling problems. A connection of polyominoes to commutative algebra has first been established by the second author of this paper by assigning to each polyomino its ideal of inner minors, see \cite{Q}. This class of ideals widely generalizes the ideal of $2$-minors of a matrix of indeterminates, and even that of the  ideal of 2-minors of two-sided ladders. It also includes the meet-join ideal of plane distributive lattices. Those classical ideals have been extensively studied in the literature, see for example \cite{BV}, \cite{C}   and \cite{HT}. Typically one determines for such ideals their Gr\"obner bases, determines their resolution and computes their regularity, checks whether the rings defined by them are normal, Cohen-Macaulay or Gorenstein. A first step in this direction for the inner minors of a polyomino (also called polyomino ideals) has been done by Qureshi in the afore mentioned paper. Very recently those  convex polyominoes have been classified in \cite{EHH} whose ideal of inner minors is linearly related or has a linear resolution. For  some  special polyominoes  also  the regularity of the ideal of inner minors is known, see \cite{ERQ}.

In this paper, for balanced polyominoes, we provide some positive answers to the questions addressed before. To define a balanced polyomino, one labels the  vertices of a polyomino by integer numbers in a way that row and column sums are zero along intervals that belong to the polyomino. Such a labeling is called admissible. There is a natural way to attach to each admissible labeling of a polyomino $\Pc$ a binomial. The given polyomino is called balanced if all the binomials arising from admissible labelings belong to the ideal of inner minors $I_\Pc$ of $\Pc$. It turns out that $\Pc$ is balanced  if and only if $I_\Pc$  coincides with the lattice ideal determined by $\Pc$, see Proposition~\ref{ilambda}. This is the main observation of Section~\ref{innerminors} where we provide the basic definitions regarding polyominoes and their ideals of inner minors. An important consequence of Proposition~\ref{ilambda} is the result stated in Corollary~\ref{primep} which asserts that for any balanced polyomino $\Pc$, the ideal $I_\Pc$ is a prime ideal and that its  height coincides with the number of cells of $\Pc$. It is conjectured in \cite{Q} by the second author of this paper that $I_\Pc$ is a prime ideal for any simple polyomino $\Pc$. A polyomino is called simple if it has no `holes', see Section~\ref{innerminors} for the precise definition. We expect that simple polyominoes are balanced. This would then imply  Qureshi's conjecture.

In Section 2 we identify the primitive binomials of a balanced polyomino (Theorem~\ref{primitive}) and deduce from this that for a balanced polyomino $\Pc$ the ideal $I_\Pc$ has a squarefree initial ideal for any monomial order. This then implies, as shown in Corollary~\ref{balancedcm}, that the residue class ring of $I_\Pc$ is a normal Cohen-Macaulay domain. Finally, in Section~\ref{classes} we show that all row or column convex, and all tree-like polyminoes are  simple and  balanced. This supports our conjecture that simple polyominoes are balanced.

\section{The ideal of inner minors of a polyomino}
\label{innerminors}
In this section we introduce polyominoes and their ideals of inner minors. Given $a=(i,j)$ and $b=(k,l)$ in $\NN^2$ we write  $a\leq b$ if $i\leq k$ and $j\leq l$. The set $[a,b]=\{c\in\NN^2\:\; a\leq c\leq b\}$ is called an {\em interval}, and an interval of the from $C=[a,a+(1,1)]$ is called a {\em cell} (with left lower corner $a$). The elements of $C$ are called the {\em vertices} of $C$, and the sets $\{a,a+(1,0)\}, \{a,a+(0,1)\},  \{a+(1,0),  a+(1,1)\}$ and   $\{a+(0,1),  a+(1,1)\}$ the {\em edges} of $C$.


Let $\Pc$ be a finite collection of cells of $\NN^2$, and let $C$ and $D$ be two cells of $\Pc$. Then $C$ and $D$ are said to be {\em connected}, if there is a sequence of cells $C= C_1, \ldots, C_m =D$ of $\Pc$  such that $C_i \cap C_{i+1}$ is an edge of $C_i$ for $i=1, \ldots, m-1$. If in addition, $C_i \neq C_j$ for all $i \neq j$, then $\mathcal{C}$ is called a {\em path} (connecting $C$ and $D$). The collection of cells $\Pc$ is called a {\em polyomino} if any two cells of
 $\Pc$ are connected, see Figure~\ref{polyomino}.

 \begin{figure}[htbp]
\begin{center}
\includegraphics[width =3cm]{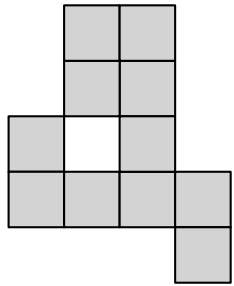}
\end{center}
\caption{A polyomino}\label{polyomino}
\end{figure}

Let $\Pc$ be a polyomino, and let $K$ be a field. We denote by $S$ the polynomial over $K$ with variables $x_{ij}$ with $(i,j)\in V(\Pc)$.  Following \cite{Q} a $2$-minor $x_{ij}x_{kl}-x_{il}x_{kj}\in S$ is called an {\em inner minor} of $\Pc$ if all the cells $[(r,s),(r+1,s+1)]$ with $i\leq r\leq k-1$ and $j\leq s\leq l-1$ belong to $\Pc$. In that case the interval $[(i,j),(k,l)]$ is called an {\em inner interval} of $\Pc$. The ideal $I_\Pc\subset S$ generated by all inner minors of $\Pc$ is called the {\em polyomino ideal} of $\Pc$. We also set $K[\Pc]=S/I_\Pc$.

Let $\Pc$ be a polyomino. An interval $[a,b]$ with $a=(i,j)$ and $b=(k,l)$ is called a {\em horizontal  edge interval} of $\Pc$  if $j=l$ and  the sets $\{r,r+1\}$ for $r=i,\ldots,k-1$ are edges of cells of $\Pc$. Similarly one defines vertical edge intervals of $\Pc$.  According to \cite{Q},  an integer value function $\alpha\: V(\Pc) \to \ZZ$ is called {\em  admissible}, if for all maximal horizontal or vertical edge intervals $\Ic$ of $\Pc$ one has
\[
\sum_{a\in \Ic}\alpha(a)=0.
\]

 \begin{figure}[htbp]
\begin{center}
\includegraphics[width =3.3cm]{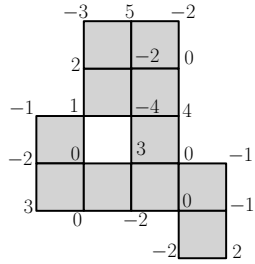}
\end{center}
\caption{An admissible labeling}\label{labeling}
\end{figure}

In Figure~\ref{labeling} an admissible labeling of the polyomino displayed Figure~\ref{polyomino} is shown. Given an admissible labeling $\alpha$ we define the binomial \[
f_{\alpha} = \prod_{a \in V(\Pc) \atop \alpha(a) > 0} x_a^{\alpha(a)} - \prod_{a \in V(\Pc) \atop \alpha(a) < 0} x_a^{-\alpha(a)} ,
 \]
Let $J_\Pc$ be the ideal generated by the binomials $f_\alpha$ where $\alpha$ is an admissible labeling of $\Pc$. It is obvious that $I_\Pc\subset J_\Pc$.
We call a polyomino {\em balanced} if for any admissible labeling $\alpha$, the binomial $f_{\alpha} \in I_{\Pc}$. This is the case if and only if $I_\Pc=J_\Pc$.

Consider the free abelian group $G=\Dirsum_{(i,j)\in V(\Pc)}\ZZ e_{ij}$ with basis elements $e_{ij}$. To any cell $C=[(i,j),(i+1,j+1)]$  of $\Pc$ we attach the element $b_C=e_{ij}+e_{i+1,j+1}- e_{i+1,j}-e_{i,j+1}$ in $G$ and let $\Lambda\subset G$ be the lattice spanned by these elements.

\begin{Lemma}
\label{lattice}
The elements $b_C$ form a $K$-basis of $\Lambda$ and hence $\rank_\ZZ \Lambda=|\Pc|$. Moreover,  $\Lambda$ is saturated. In other words, $G/\Lambda$ is torsionfree.
\end{Lemma}

\begin{proof}
 We order the basis elements $e_{ij}$ lexicographically. Then the lead term of $b_C$ is $e_{ij}$. This shows that the elements $b_C$ are linearly independent and hence form a $\ZZ$-basis of $\Lambda$. We may complete this basis of $\Lambda$ by the elements $e_{ij}$ for which $(i,j)$ is not a left lower corner  of a cell of $\Pc$ to obtain a basis of $G$. This shows that $G/\Lambda$ is free, and hence torsionfree.
\end{proof}

The lattice ideal $I_\Lambda$ attached to the lattice $\Lambda$ is the ideal generated by all binomials
\[
f_v= \prod_{a \in V(\Pc) \atop v_a > 0} x_a^{v_a} - \prod_{a \in V(\Pc) \atop v_a < 0} x_a^{-v_a}
\]
with $v\in \Lambda$.

\begin{Proposition}
\label{ilambda}
Let $\Pc$ be a balanced polyomino. Then $I_\Pc= I_\Lambda$.
\end{Proposition}

\begin{proof}
The assertion  follows once we have shown that for any $v\in \Lambda$ there exists an admissible labeling $\alpha$ of $\Pc$ such that $v_a=\alpha(a)$ for all $a\in V(\Pc)$. Indeed, since the elements $b_C\in \Lambda$ form a $\ZZ$-basis  of $\Lambda$,  there exist integers $z_C\in \ZZ$ such that $v=\sum_Cz_Cb_C$. We set $\alpha=\sum_{C\in \Pc}z_C\alpha_C$ where for $C=[(i,j),(i+1,j+1)]$,
\[
\alpha_C((k,l))= \left\{ \begin{array}{ll}
          1, &  \text{if  $(k,l)=(i,j)$ or $(k,l)=(i+1,j+1)$}, \\
          -1, &\text{if  $(k,l)=(i+1,j)$ or $(k,l)=(i,j+1)$}, \\
           0,& \text{otherwise}.
        \end{array} \right.
\]
Then $\alpha(a)=v_a$ for all $a\in V(\Pc)$. Since each $\alpha_C$ is an admissible labeling of $\Pc$ and since any linear combination of admissible labelings is  again   an admissible labeling, the desired result follows.
\end{proof}

\begin{Corollary}
\label{primep}
If $\Pc$ is a balanced polyomino, then $I_\Pc$ is a prime ideal of height $|\Pc|$.
\end{Corollary}

\begin{proof}
By Proposition~\ref{ilambda}, $I_\Pc=I_\Lambda$ and by Lemma~\ref{lattice}, $\Lambda$ is saturated. It follows that $I_\Pc$ is a prime ideal, see  \cite[Theorem 7.4]{MS}. Next if follows from \cite[Corollary 2.2]{ES} (or \cite[Proposition 7.5]{MS})  that $\height I_\Pc=\rank_\ZZ \Lambda$. Hence the desired conclusion follows from Lemma~\ref{lattice}.
\end{proof}

Let $\Pc$ be a polyomino and let $[a,b]$ an interval with the property that $\Pc\subset [a,b]$.  According to \cite{Q}, a polyomino $\Pc$  is called {\em simple}, if for any cell $C$ not belonging to $\Pc$ there exists a path $C=C_1,C_2,\ldots,C_m=D$ with  $C_i\not \in \Pc$ for $i=1,\ldots,m$ and such that $D$ is not a cell of  $[a, b]$. It is conjectured in \cite{Q} that $I_\Pc$ is a prime ideal if $\Pc$ is simple. There exist  examples of polyominoes for which $I_\Pc$ is a prime ideal but which are not simple. Such an example is shown in Figure~\ref{prime1}. On the other hand, we conjecture that a polyomino is simple if and only it is balanced. This conjecture implies Qureshi's conjecture on simple polyominoes.

 \begin{figure}[htbp]
\begin{center}
\includegraphics[width =2cm]{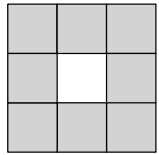}
\end{center}
\caption{Not simple but prime}\label{prime1}
\end{figure}

\section{Primitive binomials of balanced polyominoes}
\label{primitivesection}
The purpose of this section is to identify for any balanced polyomino $\Pc$ the  primitive binomials in $I_\Pc$. This will allow us to show that the initial ideal of $I_\Pc$  is a squarefree monomial ideal for any monomial order.

The primitive binomials in $\Pc$ are determined by cycles. A sequence of  vertices $\Cc= a_1,a_2, \ldots, a_m$ in  $V(\Pc)$ with $a_m = a_1$ and such that $a_i \neq a_j$ for all  $1 \leq i < j \leq m-1$ is a called a {\em cycle}  in $\Pc$ if the following conditions hold:
\begin{enumerate}
\item[(i)]  $[a_i, a_{i+1}] $ is a horizonal or vertical edge interval of $\Pc$ for all $i= 1, \ldots, m-1$;
\item[(ii)] for $i=1, \ldots, m$ one has: if $[a_i, a_{i+1}]$ is a horizonal interval of $\Pc$, then  $[a_{i+1}, a_{i+2}]$  is a vertical edge interval of $\Pc$ and vice versa. Here,  $a_{m+1} = a_2$.
\end{enumerate}


\begin{figure}[htbp]
  \begin{center}
    \begin{tabular}{c}

      \begin{minipage}{0.5\hsize}
        \begin{center}
          \includegraphics[clip, width=3cm]{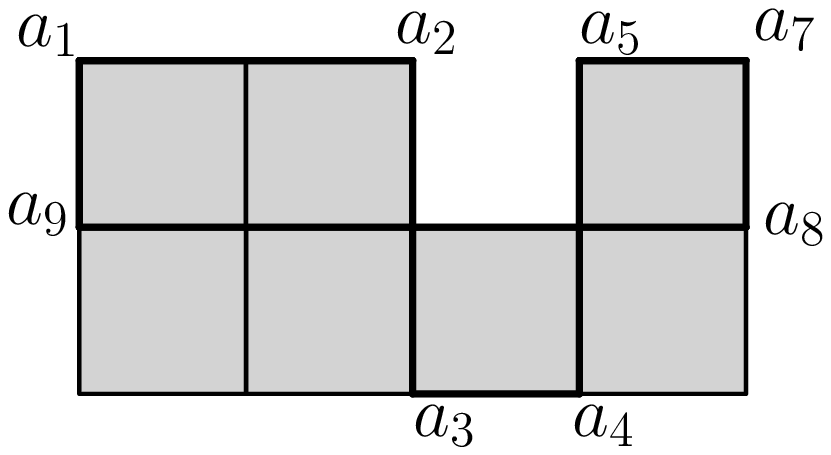}
        \end{center}
      \end{minipage}

      \begin{minipage}{0.5\hsize}
        \begin{center}
          \includegraphics[clip, width=3cm]{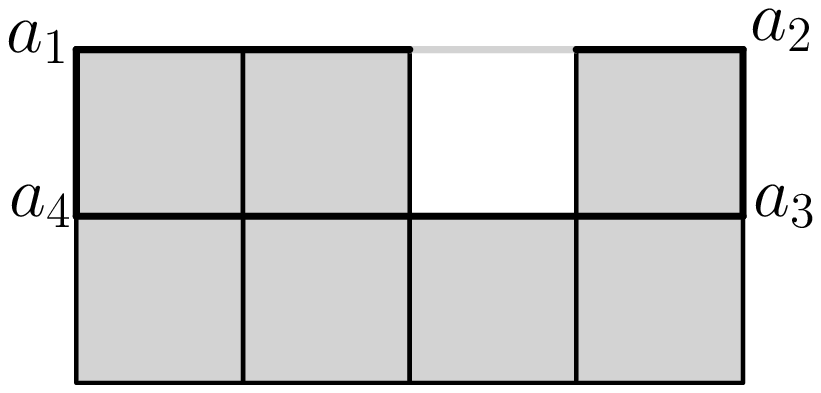}
        \end{center}
      \end{minipage}

    \end{tabular}
    \caption{A cycle and a non-cycle in $\Pc$}
    \label{figlena}
  \end{center}

\end{figure}

It follows immediately from the definition of a cycle that $m-1$ is even. Given a cycle $\Cc$, we attach to $\Cc$ the binomial
\[
f_{\Cc} = \prod_{i=1}^{(m-1)/2} x_{a_{2i-1}} - \prod_{i=1}^{(m-1)/2} x_{a_{2i}}
\]

\begin{Theorem} \label{primitive}
Let $\Pc$ be a balanced  polyomino.
\begin{enumerate}
\item[{\em (a)}] Let $\Cc$ be a cycle in $\Pc$. Then $f_\Cc\in I_\Pc$.
\item[{\em (b)}] Let $f \in I_{\Pc}$ be a primitive binomial. Then there exists a cycle $\Cc$ in $\Pc$ such that each maximal interval of $\Pc$ contains at most two vertices of $\Cc$ and $f=\pm f_{\Cc}$.
\end{enumerate}
\end{Theorem}

\begin{proof}
(a) Let $\Cc= a_1,a_2, \ldots, a_m$ be the cycle in $\Pc$. We define the  labeling $\alpha$ of $\Pc$ by setting $\alpha(a)=0$ if $a\not\in \Cc$  and $\alpha(a_i)=(-1)^{i+1}$ for $i=1,\ldots,m$, and claim that $\alpha$ is an admissible labeling of $\Pc$. To see this we consider a maximal horizontal edge interval  $I$ of $\Pc$. If $I\sect \Cc=\emptyset$, then $\alpha(a)=0$ for all $a\in I$. On the other hand, if $I\sect\Cc\neq \emptyset$, then there exist integers $i$ such that $a_i,a_{i+1} \in I$ (where $a_{i+1}=a_1$ if $i=m-1$), and no other vertex of $I$ belongs $\Cc$. It follows that $\sum_{a\in I}\alpha(a)=0$. Similarly, we see that $\sum_{a\in I}\alpha(a)=0$ for any vertical edge interval. It follows form the definition of $\alpha$ that $f_\Cc=f_{\alpha}$, and hence since $\Pc$ is balanced it follows that $f_\Cc\in I_\Pc$.

(b) Let $f \in I_\Pc$ be a primitive binomial. Since $\Pc$ is balanced and $f$ is irreducible, \cite[Theorem 3.8(a)]{Q} implies that there exists an admissible labeling $\alpha$ of $\Pc$ such that
 \[
f=f_{\alpha} = \prod_{a \in V(\Pc) \atop \alpha(a) > 0} x_a^{\alpha(a)} - \prod_{a \in V(\Pc) \atop \alpha(a) < 0} x_a^{-\alpha(a)}.
 \]
Choose $a_1 \in V(\Pc)$ such that $\alpha (a_1) >0$. Let $I_1$ be the maximal horizontal edge interval with $a_1 \in I_1$. Since  $\alpha$ is admissible, there exists some $a_2 \in I_1$ with $\alpha(a_2) < 0$. Let $I_2$ be the maximal vertical edge interval containing $a_2$. Then  similarly as before, there exists  $a_3 \in I_2$ with $\alpha(a_3)>0$.  In the next step we consider the maximal horizontal edge interval containing and $a_3$ and proceed as before.  Continuing in this way we obtain a sequence $a_1,a_2,a_3.\ldots,$ of vertices of $\Pc$  such that  $\alpha(a_1), \alpha(a_2), \alpha(a_3),\ldots$ is a sequence with alternating signs.  Since $V(\Pc)$ is a finite set, there exist a number $m$ such that $a_i\neq a_j$ for all $1\leq i<j\leq m$  and $a_m= a_i$ for some $i<m$. If follows  that $\alpha(a_m)=\alpha(a_i)$ which implies that $m-i$ is even. Then the sequence $\Cc=a_i,a_{i+1},\ldots a_m$ is a cycle in $\Pc$, and hence by (a), $f_\Cc\in I_\Pc$.

For any binomial $g=u-v$ we set  $g^{(+)}=u$ and $g^{(-)}=v$. Now
if $i$ is odd, then $f_\Cc^{(+)}$ divides $f^{(+)}$ and  $f_\Cc^{(-)}$ divides $f^{(-)}$, while if $i$ is even, then   $f_\Cc^{(+)}$ divides $f^{(-)}$ and $f_\Cc^{(-)}$ divides $f^{(+)}$. Since $f$ is primitive, this implies that $f=\pm f_{\Cc}$, as desired.
\end{proof}

\begin{Corollary}
Let $\Pc$ be a balanced polyomino. Then $I_{\Pc}$ admits a squarefree initial ideal for any monomial order.
\end{Corollary}
\begin{proof}
By Corollary~\ref{primep},  $I_\Pc$ is a prime ideal, since  $\Pc$  is a balanced polyomino. This implies that $I_\Pc$ is a toric ideal, see for example \cite[Theorem 5.5]{EH}.  Now we use the fact (see \cite[Lemma 4.6]{St} or \cite[Corollary 10.1.5]{HHBook}) that the primitive binomials of a toric ideal form a universal Gr\"obner basis. Since by Theorem~\ref{primitive}, the primitive binomials of $I_\Pc$  have squarefree initial terms for any monomial order, the desired conclusion follows.
\end{proof}

\begin{Corollary}
\label{balancedcm}
Let $\Pc$ be a balanced  polyomino. Then $K[\Pc]$ is a normal Cohen-Macaulay domain of dimension $|V(\Pc)|-|\Pc|$.
\end{Corollary}

\begin{proof}
A toric ring whose toric ideal admits a squarefree initial ideal is normal by theorem of Sturmfels \cite[Chapter 8]{St}, and by a theorem of Hochster (\cite[Theorem 6.3.5]{BH}) a normal toric ring is Cohen--Macaulay. Since $\Pc$ is balanced, we know from Proposition~\ref{ilambda} that $I_\Pc=I_\Lambda$ where $\Lambda$ is the lattice in $\Dirsum_{(i,j)\in V(\Pc)}\ZZ e_{ij}$ spanned by the elements $b_C=e_{ij}+e_{i+1,j+1}- e_{i+1,j}-e_{i,j+1}$ where $C=[(i,j),(i+1,j+1)]$ is a cell of $\Pc$. By \cite[Proposition 7.5]{MS}, the height of $I_\Lambda$ is equal to the rank of $\Lambda$. Thus we see that $\height I_\Pc=|\Pc|$. It follows that
the Krull dimension of $K[\Pc]$ is equal to $|V(\Pc)|-|\Pc|$, as desired.
\end{proof}

\section{Classes of balanced  polyominoes}
\label{classes}
In this section we consider two classes of balanced polyominoes. As mentioned in Section~\ref{innerminors} one expects  that any simple polyomino is balanced. In this generality we do not yet have a proof of this statement. Here we want to consider only two special classes of polyominoes  which are simple and balanced, namely the  row and column convex polyominoes, and the tree-like polyominoes.

Let $\Pc$ be a  polyomino. Let $C=[(i,j),(i+1,j+1)]$ be a cell of $\Pc$. We call  the vertex $a=(i,j)$ the {\em left lower corner} of $C$. Let  $C_1$ and $C_2$ be two cells with left lower corners $(i_1,j_1)$ and $(i_2,j_2)$, respectively.  We say that $C_1$ and $C_2$  are  in {\em horizontal (vertical) position} if $j_1=j_2$ ($i_1=i_2$). The polyomino $\Pc$ is called {\em row convex} if for any two horizontal cells $C_1$ and $C_2$  with lower left corners $(i_1,j)$ and $(i_2,j)$ and $i_1<i_2$, all cells with lower left corner $(i,j)$ with $i_1\leq i\leq i_2$ belong to $\Pc$. Similarly one defines {\em column convex} polyominoes. For example, the polyomino displayed in Figure~\ref{figcolumnconvex} is column convex but not row convex.

        \begin{figure}[htbp]
        \begin{center}
          \includegraphics[clip, width=2cm]{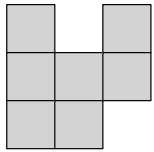}
    \caption{Column convex but not row convex}
     \label{figcolumnconvex}
     \end{center}
     \end{figure}

A {\em neighbor} of a cell $C$ in $\Pc$ is a cell $D$ which shares a common edge with $C$. Obviously, any cell can have at most four neighbors. We call a cell of $\Pc$ a {\em leaf}, if it has an edge which does not has a  common vertex with any other cell.  Figure~\ref{leaves}  illustrates this concept.

       \begin{figure}[htbp]
        \begin{center}
          \includegraphics[clip, width=4.3cm]{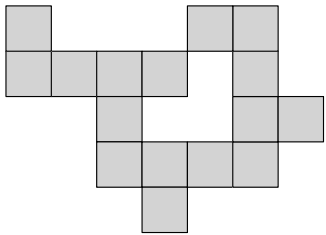}
  \caption{A polyomino with three leaves}\label{leaves}
     \end{center}
     \end{figure}

The polyomino $\Pc$ is called {\em tree-like}  each subpolyomino of $\Pc$  has a leaf. The polyomino displayed in Figure~\ref{leaves} is not tree-like, because it contains a subpolyomino which has no leaf. On the other hand, Figure~\ref{treelike} shows a tree-like polyomino.

\begin{figure}[htbp]
        \begin{center}
          \includegraphics[clip, width=4.3cm]{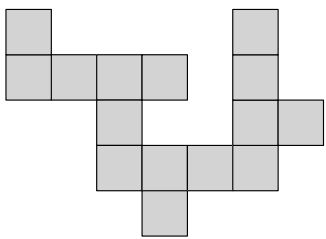}
  \caption{A tree-like polyomino}\label{treelike}
     \end{center}
     \end{figure}

A free vertex of $\Pc$ is a vertex which belongs to exactly one cell. Notice that any leaf has two free vertices.

We call a path of cells a {\em  horizontal (vertical) cell interval}, if the left lower corners of the path form a horizontal (vertical) edge  interval.
Let $C$ be a leaf, and let $\Ic$ by the maximal cell interval to which $C$ belongs, and assume that   $\Ic$ is a horizontal (vertical) cell interval. Then we call $C$ a {\em good leaf}, if for one of the free vertices of $C$ the maximal horizontal (vertical) edge interval which contains it  has the same length as $\Ic$. We call a leaf  {\em bad} if it is not good, see Figure~\ref{badcells}.


\begin{Theorem}
\label{whatweunderstand}
Let $\Pc$ be a row or column convex, or a tree-like polyomino. Then $\Pc$ is balanced and simple.
 \begin{figure}[htbp]
\begin{center}
\includegraphics[width =2cm]{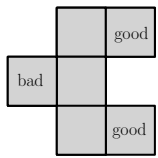}
\caption{Bad and good leaves}\label{badcells}
\end{center}
\end{figure}

\end{Theorem}

\begin{proof}

Let $\Pc$ be a tree-like polyomino. We first show that $\Pc$ is balanced. Let $\alpha$ be an admissible labeling of $\Pc$. We have to show that $f_\alpha\in I_\Pc$. To prove this we first show that $\Pc$ has a good leaf.

If $|\Pc| = 1$, then the assertion is trivial. We may assume that $|\Pc| \ge 2$. Let
\[
n_i = |\{C \in \Pc : \deg C=i \}|.
\]
Observe that $\sum_{C \in \Pc} \deg C  = n_1 + 2n_2 + 3n_3 + 4n_4$, where $\deg C$ denotes the number of neighbor cells of $\Pc$.

Let $\Pc$ be any polyomino with cells $C_1, \ldots, C_n$. Associated to $\Pc$, is the so-called {\em connection graph} on vertex $[n]$ with the edge set
$ \{\{i,j\}: E(C_i) \cap E(C_j) \neq \emptyset \}$.

It is easy to see that connection graph  of a tree-like polyomino is a tree. Therefore, using some elementary facts from graph theory, we obtain that

\[
n_1 + 2n_2 + 3n_3 + 4n_4 = 2 (|\Pc| -1) = 2(n_1 + n_2 + n_3 + n_4-1).
\]

This  implies that  $n_1 = n_3+2n_4+2$.
Let $g(\Pc)$ be the number of good leaves in $\Pc$ and $b(\Pc)$ be the number of bad leaves in $\Pc$.
It is obvious that $n_1= g(\Pc) + b(\Pc)$.
Then we have
\begin{eqnarray}
\label{shikama}
g(\Pc) = n_3+ 2n_4 +2 -b(\Pc).
\end{eqnarray}
Next, we show that $b(\Pc) \leq n_3$. Suppose that $C$ is a  bad leaf in $\Pc$ and $\Ic$ the unique maximal cell interval to which $C$ belongs.
Let $D_C$ be the end cell of the interval $\Ic$. Observe that $C\neq D_C$.
We claim that $\deg D_C = 3$.  Indeed, since $C$ is bad,  the length of the maximal intervals containing the two free vertices of $C$  is bigger than the length of the interval $\Ic$.  See Figure \ref{Thm1} where the cells belonging to $\Ic$ are marked with dots and $E$ is the  cell next to  $D_C$ which does not belong to $\Pc$. Since $E \notin\Pc$,  the cells $D_1$ and $D_2$ belong to $\Pc$. Suppose $D_3\not\in \Pc$. Since $\Pc$ is a polyomino there exists a  path $\Cc$ connecting $D_1$ and $D_C$. Since $E,D_3\not\in \Pc$ the path $\Cc$ (which is a subpolyomino of $\Pc$)  does not have a leaf, contradicting the assumption that $\Pc$ is tree-like. Therefore, $D_3\in\Pc$. Similarly one shows that $D_4\in \Pc$. This shows that $\deg D_C=3$.

 \begin{figure}[htbp]
\begin{center}
\includegraphics[width =2.8cm]{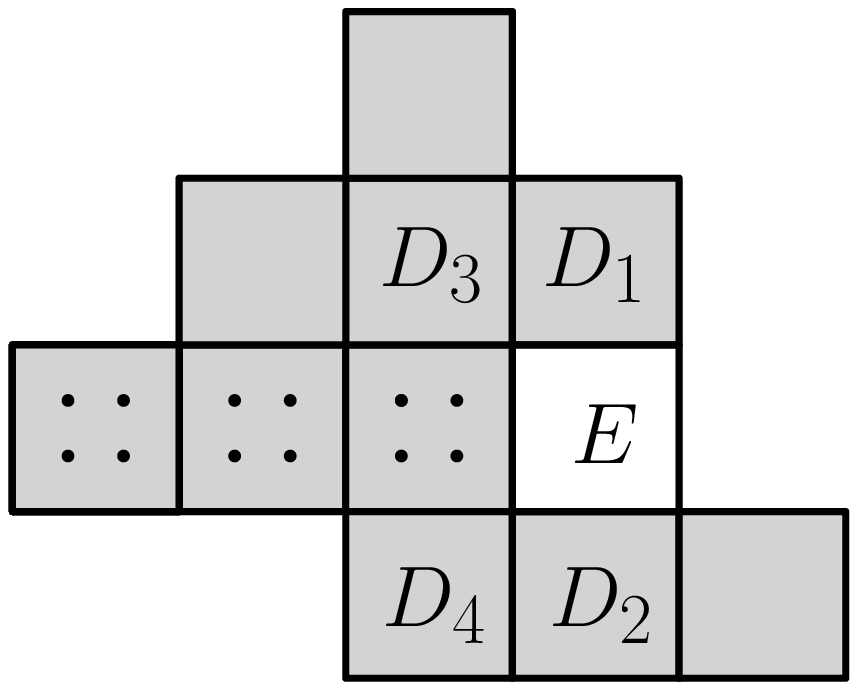}
\caption{}\label{Thm1}
\end{center}
\end{figure}

Moreover, $D_C$ cannot be the end cell of any other cell interval, because $D_3$ and $D_4$ are neighbors of $D_C$. Thus we obtain a one-to-one correspondence between the bad leafs $C$ and the  cells $D_C$ as defined before. It follows that $n_3\leq b(\Pc)$. Therefore, by (\ref{shikama}) we obtain
\[
g(\Pc)=n_3+ 2n_4 +2 -b(\Pc)\geq 2n_4+2\geq 2.
\]
Thus, we have at least $2$ good end cell for every tree-like polyomino.

Now we show $\Pc$ is balanced. Indeed, let $\alpha$ be an admissible labeling of $\Pc$. We want to show that $f_\alpha\in I_\Pc$.  Let $C$ be a good leaf of $\Pc$ with free vertices $a_1$ and $a_2$. Then $\alpha(a_1)=-\alpha(a_2)$. If $\alpha(a_i)=0$ for $i=1,2$, then $\alpha$ restricted to $\Pc'$ is an admissible labeling of  $\Pc'$, where $\Pc'$ is obtained from $\Pc$ by removing the cell $C$ from $\Pc$. Inducting on  the number of cells we may then assume that $f_\alpha\in I_{\Pc'}$. Since $I_{\Pc'}\subset I_\Pc$, we are done is this case.

Assume now that $\alpha(a_1)\neq 0$. We proceed by induction on $|\alpha(a_1)|$. Note that  $\alpha(a_2)\neq 0$, too. Since $C$ is a good leaf, we may assume that  the maximal interval $[a_2,b]$ to which $a_2$ belongs has the same length as the cell interval $[C,D_C]$ and $\alpha(a_1)>0$. Then $\alpha(a_2)<0$ and hence, since $\alpha$ is admissible,  there exists a $c\in [a_2,b]$  with  $\alpha(c)>0$. The cells of the interval $[c,a_1]$ all belong to $[C,D_C]$. Therefore, $g=x_cx_{a_1}-x_dx_{a_2}\in I_\Pc$. Here $d$ is the vertex  as indicated in Figure~\ref{Thm2}.

 \begin{figure}[htbp]
\begin{center}
\includegraphics[width =3.3cm]{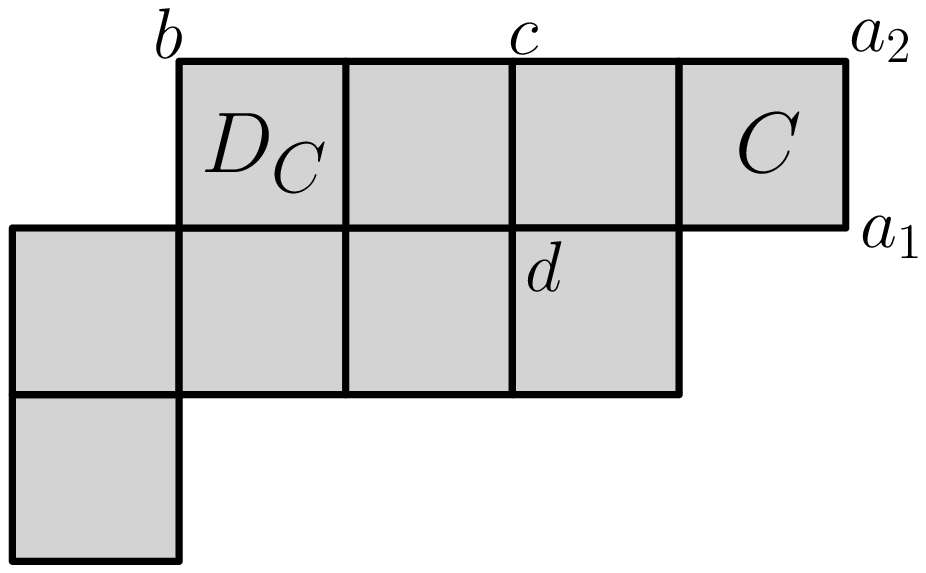}
\caption{}\label{Thm2}
\end{center}
\end{figure}

Notice that the labeling $\beta$ of $\Pc$ defined by
\[
\beta_1(e)= \left\{ \begin{array}{ll}
          \alpha(e)-1, &  \text{if  $e=a_1$ or $e=c$}, \\
         \alpha(e)+1, &\text{if  $e=a_2$ or $e=d$}, \\
           \alpha(e),& \text{elsewhere}.
        \end{array} \right.
\]
is admissible and $|\beta(a_1)| < |\alpha( a_1)|$. By induction hypothesis we have that $f_\beta\in I_\Pc$.
Then the following relation
\begin{eqnarray}\label{relation}
f_{\alpha} -( f_{\alpha}^{(+)}/ x_c x_{a_1}) g = (x_{a_2}x_d)f_{\beta}
\end{eqnarray}
gives that $f_{\alpha} \in I_{\Pc}$, as well.


Next, we show that $\Pc$ is simple by applying induction on number of cells. The assertion is trivial if $\Pc$ consists of only one cell.
Suppose now  $|\Pc|>1$,  and let $D$ be a cell not belonging to $\Pc$ and $\Ic$ an interval such  $V(\Pc)\subset \Ic$.

Since  $\Pc$ is tree-like, $\Pc$ admits a leaf cell $C$. Let $\Pc'$ be the polyomino which is obtained from $\Pc$ by removing the cell $C$. Then $\Pc'$ is again tree-like and hence is simple by induction. Therefore,  since $D\not\in\Pc'$, there exists  a path $\Dc': D=D_1,D_2,\cdots, D_m$ of cells with $D_i\not\in \Pc'$ for all $i$ and such that $D_m$ is a border cell of $\Ic$. We let $\Dc=\Dc'$, if $D_i\neq C$ for all $i$. Note that $C\neq D_1, D_m$. Suppose $D_i=C$ for some $i$ with $1<i<m$. Since $C$ a leaf, it follows that the cells $C_1,C_2,C_3,C_4$ and $C_5$ as shown in Figure~\ref{dom} do not belong to $\Pc$.

\begin{figure}[htbp]
\begin{center}
\includegraphics[width =2.3cm]{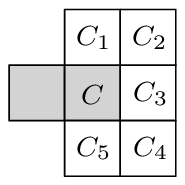}
\caption{}\label{dom}
\end{center}
\end{figure}

Since $D_i=C$ and since $\Dc'$ is a path of cells it follows that $D_{i-1}\in \{C_1,C_3,C_5\}$ and $D_{i+1}\in \{C_1,C_3,C_5\}$. If $D_{i-1}=C_1$ and $D_{i+1}=C_3$, then we let
\[
\Dc: D_1,\ldots, D_{i-1},C_2,D_{i+1},\ldots, D_m,
\]
and if $D_{i-1}=C_1$ and $D_{i+1}=C_5$, then we let
\[\Dc: D_1,\ldots, D_{i-1},C_2,C_3,C_4,D_{i+1},\ldots, D_m.
 \]
Similarly one defines  $\Dc$ in all the other cases that my occur for $D_{i-1}$ and $D_{i+1}$ , so that in any case $\Dc'$ can be replaced by the path of cells $\Dc$ which does not contain any cell of $Pc$ and connects $D$ with the border of $\Ic$. This shows that $\Pc$ is simple.

\medskip
Now we show that if $\Pc$ is row or column convex then $\Pc$ is balanced and simple. We may assume that $\Pc$ be column convex. We first show that $\Pc$ is balanced. For this part of the proof we follow the arguments given in \cite[Proof of Theorem 3.10]{Q}.  Let $\Cc_1=[C_1, C_n]$ be the left most column interval of $\Pc$ and $a$ be lower left corner of $C_1$. Let $\alpha $ be an admissible labeling for $\Pc$. If $\alpha (a) =0$, then $\alpha$ restricted to $\Pc'$ is an admissible labeling of  $\Pc'$, where $\Pc'$ is obtained from $\Pc$ by removing the cell $C_1$ from $\Pc$. By applying induction on  the number of cells we have that $f_\alpha\in I_{\Pc'} \subset I_\Pc$, and we are done is this case.

Now we may assume that $\alpha(a) \neq 0$. We may assume that $\alpha (a) > 0$. By following the same arguments as in case of tree-like polyominoes, it suffice to show that there exist an inner interval $[(i,j), (k,l)]$ with $i<k$ and $j<l$ such that $\alpha((i,j))\alpha((k,l)) >0$ or $\alpha((k,j))\alpha((i,l)) >0$.

Let $[a,b]$ and $[a,c]$ be the maximal horizontal and vertical edge intervals of $\Pc$ containing $a$. Then there exist $e \in [a, b]$ and $f \in  [a, c]$ such that $\alpha(e), \alpha(f) < 0$ because $\alpha$ is admissible. Let $[f, g]$ be the maximal horizontal edge interval of $\Pc$ which contains $f$. Then there exists a vertex $h \in [f, g]$ such that $\alpha(h) > 0$. If $\size([f, g])\leq  \size([a, b ])$, then column convexity of $\Pc$ gives that $ [a, h]$ is an inner interval of P and $\alpha(a)\alpha(h)>0$. Otherwise, if  $\size([f,g]) \geq  \size([a,b])$, then again by using column convexity of $\Pc$, $x_a x_q - x_e x_f \in I_{\Pc}$ where $q$ is as shown in following figure. By following the same arguments as in case of tree-like polyominoes, we conclude that $f_\alpha \in I_\Pc$.

Now we will show that $\Pc$ is simple. We may assume that $V(\Pc)\subset [(k,l),(r,s)]$  with $l>0$. Let $C\not\in \Pc$ be a cell with with lower left corner $a=(i,j)$, and consider the infinite path $\Cc$ of cells where $\Cc=C_0,C_1,\ldots $ and where the lower left corner of $C_k$ is  $(i,k)$ for $k=0,1,\ldots$. If $\Cc\sect\Pc=\emptyset$, then $C$ is connected to $C_0$. On the other hand, if $\Cc\sect\Pc\neq \emptyset$, then, since $\Pc$ is column convex,  there exist integers $k_1,k_2$ with $l\leq k_1\leq k_2\leq s$ such that $C_k\in \Pc$ if and only if $k_1\leq k\leq k_2$. Since $C\not\in \Pc$ it follows that $j<k_1$ or $j>k_2$. In the first case, $C$ is connected to $C_0$ and in the second case $C$ is connected to $C_s$.
\end{proof}



\begin{Corollary}
Let $\Pc$ be a row or column convex, or a tree-like polyomino. Then $K[\Pc]$ is a normal Cohen--Macaulay  domain.
\end{Corollary}

{}
 \end{document}